\documentclass[10pt]{article}
\font\smallit=cmti10
\font\smalltt=cmtt10

\usepackage{amssymb,latexsym,amsmath,epsfig,amsthm}

\makeatletter

\renewcommand\section{\@startsection {section}{1}{\z@}
{-30pt \@plus -1ex \@minus -.2ex}
{2.3ex \@plus.2ex}
{\normalfont\normalsize\bfseries\boldmath}}

\renewcommand\subsection{\@startsection{subsection}{2}{\z@}
{-3.25ex\@plus -1ex \@minus -.2ex}
{1.5ex \@plus .2ex}
{\normalfont\normalsize\bfseries\boldmath}}

\renewcommand{\@seccntformat}[1]{\csname the#1\endcsname. }

\makeatother

\newtheorem{theorem}{Theorem}[section]
\newtheorem{lemma}[theorem]{Lemma}
\newtheorem{proposition}[theorem]{Proposition}
\newtheorem{corollary}[theorem]{Corollary}

\theoremstyle{definition}
\newtheorem{definition}[theorem]{Definition}

\newtheorem{remark}[theorem]{Remark}

\begin{document}
\begin{center}
{\bf CONGRUENCES RELATING REGULAR PARTITION FUNCTIONS, A GENERALIZED TAU FUNCTION AND PARTITION FUNCTION WEIGHTED COMPOSITION SUMS}\vskip 20pt 

{\bf S. Sriram}\\
{\smallit Post Graduate and Research Department of Mathematics, National College, \\ Affiliated to Bharathidasan University, Tiruchirapalli, Tamil Nadu, India}\\
{\tt sriram.priya02@yahoo.com}\\
\vskip 20pt

{\bf A. David Christopher\footnote{Corresponding author}}\\
{\smallit Department of Mathematics, The American College, Madurai, Tamil Nadu, India.}\\
{\smallit Post Graduate and Research Department of Mathematics, National College,\\ Affiliated to Bharathidasan University, Tiruchirapalli, Tamil Nadu, India}\\
{\tt davidchristopher@americancollege.edu.in}\\
\vskip 20pt
\end{center}

\vskip 20pt
\centerline{\smallit Received: , Accepted: , Published: }
\vskip 30pt

\centerline{\bf Abstract}

\noindent
Let $n$ and $t$ be positive integers with $t\geq 2$. Let $R_t(n)$ be the number of $t$-regular partitions of $n$. 
A class of functions, denoted $\tau_k(n)$, is defined as follows:
\[
q\prod_{m=1}^{\infty}(1-q^m)^k=\sum_{n=1}^{\infty}\tau_k(n)q^n,
\]
where $k$ is an integer. We express $\tau_k(n)$ as a binomial coefficient weighted partition sum. Consequently, we obtain congruence identities that relate $\tau_k(n)$, $R_t(n)$ and partition function weighted composition sums. 

\pagestyle{myheadings}
\markright{\smalltt INTEGERS:19(2019)\hfill}
\thispagestyle{empty}
\baselineskip=12.875pt

\vskip 30pt
\section{Introduction}


Euler \cite{euler} considered the following product-to-sum representation:
\begin{equation*}
\prod_{m=1}^{\infty}(1-q^m)=\sum_{n=0}^{\infty}\omega(n)q^n,
\end{equation*}
and found that 
\begin{equation}\label{PNT}
\omega(n)=\begin{cases}
(-1)^{l}&\text{ if }n=\frac{3l^2\pm l}{2};\\
0&\text{ otherwise.}
\end{cases}
\end{equation}
This is the celebrated Euler's pentagonal number theorem.

Ramanujan \cite{ramanujan} considered the following product-to-sum representation:
\begin{equation*}
q\prod_{m=1}^{\infty}{(1-q^m)^{24}}=\sum_{n=1}^{\infty}\tau(n)q^n,
\end{equation*}
and made the following conjectures:
\begin{enumerate}
\item $\tau(nm)=\tau(n)\tau(m)$ if $\gcd(m,n)=1$,
\item for prime $p$ and integer $r\geq 1$:
$\tau(p^{r+1})=\tau(p)\tau(p^r)-p^{11}\tau(p^{r-1}),
$
\item for prime $p$:
$|\tau(p)|\leq 2p^{\frac{11}{2}}.
$
\end{enumerate}
The first two were established by Mordell \cite{mordell-1}. Delinge \cite{delinge} established the third. The function $\tau(n)$ defined above is known as Ramanujan's tau function.

The following common generalization of the aforementioned functions of Ramanujan and Euler is the object of study in this article.
\begin{definition}\label{main-definition}
Let $k\neq 0$ be an integer. We define an arithmetical function, denoted $\tau_k(n)$, in the following way:
\begin{equation}
q\prod_{m=1}^{\infty}(1-q^m)^k=\sum_{n=1}^{\infty}\tau_k(n)q^n.
\end{equation}
\end{definition}
Newman \cite{newman} and Kostant \cite{kostant} were concerned with the polynomial representation of $\tau_k(n)$. Serre \cite{serre} and Heim et al \cite{Bernard} examined  the natural density of the set $\{k\in\mathbb{N}: \tau_k(n)\neq 0\}$ at several instances of $n$.

The main objective of this paper is to explore various arithmetic properties of $\tau_k(n)$. We will first use the logarithmic derivative method to determine some arithmetic properties of $\tau_k(n)$ (at specific instances of $k$). This forms the core part of Section 2. In Section 3, $\tau_k(n)$ is expressed as a binomial coefficient weighted partition sum. As a result, congruence relations involving $\tau_k(n)$ and $t$-regular partition functions are obtained (at certain instances of $k$ and $t$). In Section 4, an expression for 
$$
\sum_{\substack{{n=a_1+a_2+\cdots a_k}\\{a_{i}\in\mathbb{N}\cup\{0\}}}}p(a_{1})p(a_{2})\cdots p(a_{k})\text{ modulo }{l}
$$
is obtained, where $l$ is an odd prime number, and $p(n)$ is the number of partitions of $n$.
\section{ Divisibility Properties of $\tau_k(n)$ using Logarithmic Differentiation}
In this section, we will discuss several congruence properties of $\tau_k(n)$ when the modulus belongs to the set $\{k-1\}\cup\{d\in\mathbb N:d|k\}$.
\subsection{ $\tau_k(n)$ Modulo $k-1$ when $k-1$ is a Prime Number}
\begin{proposition}\label{modulok-1}
Let $n$ be a positive integer, and let $k-1$ be a prime number. If $0\leq n-1-\frac{3r^2\pm r}{2}\not\equiv 0\pmod{(k-1)}$ for every non-negative integer $r$, then 
\[\tau_k(n)\equiv0\pmod{(k-1)}.
\]
\end{proposition}
\begin{proof}
The pentagonal number theorem of Euler \cite{euler} allows us to write 
\begin{align}
\notag \sum_{n=1}^{\infty}\tau_k(n)q^n&=q\prod_{m=1}^{\infty}(1-q^m)^{k}\\
&\label{tauk-rec}=\left(\sum_{r=1}^{\infty}\tau_{k-1}(r)q^r\right)\left(\sum_{s=0}^{\infty}\omega(s)q^s\right),
\end{align}
where $\omega(s)$ is as in (\ref{PNT}).
When equating the coefficients of $q^n$ at the extremes of the chain of equalities (\ref{tauk-rec}), we obtain the following identity:
\[ \label{identity1}
\sum_{i=0}^{n-1}\tau_{k-1}(n-i)\omega(i)=\tau_{k}(n).
\]
If $\tau_{k-1}(n-i)\equiv0\pmod{(k-1)}$ whenever $\omega(i)=\pm1$, then it follows that $\tau_k(n)\equiv0\pmod{(k-1)}$. So a criterion for \[\tau_{k-1}(n)\equiv0\pmod{(k-1)}\] is a requisite to proceed further. 
To that end, we define 
\[
T_{k-1}(q)=\prod_{m=1}^{\infty}(1-q^m)^{k-1}=\sum_{n=1}^{\infty}\tau_{k-1}(n)q^{n-1}.
\]
Now performing the operation $q\frac{d}{dq}(\log{T_{k-1}(q)})$ and considering the Lambert's series expansion for the sum of positive divisors of $n$ (denoted $\sigma(n)$) we obtain
\begin{equation}\label{recurrence}
n\tau_{k-1}(n+1)=-(k-1)\left(\sum_{i=1}^{n}\tau_{k-1}(i)\sigma(n+1-i)\right).
\end{equation}
Now we observe from the Identity (\ref{recurrence}) that, if $\gcd(n,k-1)=1$,
then $\tau_{k-1}(n+1)\equiv 0\pmod{(k-1)}$. Since $k-1$ is a prime number, the condition $\gcd(n,k-1)=1$ is equivalent to the condition $n\not\equiv 0\pmod{(k-1)}$. Thus, if $n\not\equiv0\pmod{(k-1)}$ then $\tau_{k-1}(n+1)\equiv0\pmod{(k-1)}$. This completes the proof.
\end{proof}
\begin{remark}
Since the $\tau$ function was introduced, an in-depth study over the value $\tau(n)$ modulo ${23}$ is an important consideration. Mordell \cite{mordell-2} gave the following criterion for the divisibility of $\tau(n)$ by 23:
\[
\tau(23n+m)\equiv0\pmod{23}
\]
for $m\in\{5, 7, 10, 11, 14, 15, 17, 19, 20, 21, 22\}$. 
Here an application of Proposition \ref{modulok-1} gives the following criterion: if $0\leq n-1-\frac{3r^2\pm r}{2}\not\equiv 0\pmod{23}$ for every non-negative integer $r$, then 
\[\tau(n)\equiv0\pmod{23}.
\]
\end{remark}
\subsection{ $\tau_k(n)$ Modulo Divisors of $k$}
Interestingly, by substituting any formal power series with integer coefficients, say $f(q)\in\mathbb{Z}[[q]]$, for $\prod\limits_{m=1}^{\infty}(1-q^m)$, one can further generalize Definition \ref{main-definition}. Denote
\[f(q)^{k}=a_0+a_1q+a_2q^2+\cdots.
\]
On differentiating with respect to $q$, we have
\[k{f(q)}^{k-1}f'(q)=a_1+2a_2q+\cdots .
\]
This gives the relation
\[na_{n}\equiv 0\pmod{|k|}.
\]
Now fixing $f(q)=\prod\limits_{m=1}^{\infty}(1-q^m)$ we have $a_n=\tau_k(n+1)$. This observation yields the following result.
\begin{proposition}\label{modulok}
Let $n$ be a positive integer, and let $k$ be an integer with $\lvert k\rvert \geq 2$. Then we have
\[
n\tau_k(n+1)\equiv 0\pmod{|k|}.
\]
\end{proposition}
The following result is a straightforward application of Proposition \ref{modulok}.
\begin{proposition}\label{modulo-divisorsofk}
Let $m\geq 0$ and $k$ be integers such that $\lvert k\rvert\geq 2$. Then we have
\[
\tau_k(|k|m+dr+1)\equiv0\left(\textrm{mod\ }{\frac{|k|}{d}}\right)
\]
for every $d\mid |k|$ such that $d<|k|$, and for every integer $r$ such that $\gcd(r,\frac{|k|}{d})=1$.
\end{proposition}
\begin{proof}
We have the congruence below based on Proposition \ref{modulok}:
\[n\tau_k(n+1)\equiv0\pmod{|k|}.
\]
Given the aforementioned congruence and $\gcd(n,|k|)=d$, it can be deduced that 
\[\frac{n}{d}\tau_k(n+1)\equiv0\left(\textrm{mod}{\frac{|k|}{d}}\right). \]
Since $\gcd(n,|k|)=d$, for a positive integer $r$ satisfying $\gcd\left(r,\frac{|k|}{d}\right)=1$, the integer $\frac{n}{d}$ must have the form $\frac{n}{d}=\frac{|k|}{d}m+r$ for some integer $m\geq 0$. As a result, $n$ assumes the form $n=|k|m+dr$. 
\end{proof}
The following list of congruences for $\tau(n)$ modulo the divisors of 24 is obtained by substituting $24$ for $k$.
\begin{proposition}
For every integer $m\geq 0$, we have
\begin{enumerate}
\item $\tau(24m+r+1)\equiv 0\pmod {24}$ for each $r\in\{1,5,7,11,13,17,19,23\}$;
\item $\tau(24m+r+1)\equiv 0\pmod{12}$ for each $r\in\{4,20\}$;
\item $\tau(24m+r+1)\equiv 0\pmod{8}$ for each $r\in\{3,9,6,15\}$;
\item $\tau(24m+r+1)\equiv 0\pmod{6}$ for each $r\in\{8,16\}$;
\item $\tau(24m+13)\equiv 0\pmod{4}$.
\end{enumerate}
\end{proposition}
\section{ Representation of $\tau_k(n)$ as a Partition Sum involving Binomial Coefficients} 
In this section, congruence properties of $\tau_k(n)$ are derived using a partition sum representation (involving binomial coefficients) of $\tau_k(n)$. Presenting the main results of this section requires the following definitions of partition theory.
\begin{definition}
\normalfont Let $n$ be a positive integer. By a {\it partition of $n$}, we mean a non-increasing sequence of positive integers whose sum equals $n$. Each element of the sequence is called a {\it part}. If each part, say $a_i$, appears $f_i$ times in a partition of $n$ then we denote that partition by $n=a_1^{f_1}\cdots a_r^{f_r}$. In this case $f_1, \cdots ,f_r$ are said to be the {\it frequencies} of the partition $a_1^{f_1}\cdots a_r^{f_r}$.
\end{definition}
\begin{definition}
\normalfont Let $n$ and $t\geq 2$ be positive integers. If all of the parts of a partition of $n$ are not divisible by $t$, then the partition is called a {\it $t$-regular partition}. We denote the number of $t$-regular partitions of $n$ by $R_t(n)$.
\end{definition}
We note that the number of partitions of $n$ with parts from the set $\mathbb{N}\setminus t\mathbb{N}$ equals the number of $t$-regular partitions of $n$, from which the following equalities arise:
\begin{align*}
\sum_{n=0}^{\infty}R_t(n)q^n&=\prod\limits_{r\in\mathbb{N}\setminus t\mathbb{N}}\frac{1}{1-q^r}\\
&=\prod_{s\in \mathbb{N}}\frac{1-q^{ts}}{1-q^{ts}}\prod\limits_{r\in\mathbb{N}\setminus t\mathbb{N}}\frac{1}{1-q^r}\\ 
&=\prod_{m=1}^{\infty}\frac{1-q^{tm}}{1-q^m}.
\end{align*}
This insight is one we utilize frequently in this section.
We express $\tau_k(n)$ as a binomial-coefficient-weighted partition sum in the following result.
\begin{theorem}\label{binomial-representation}
Let $k$ be a positive integer. We have
\begin{description}
\item{(a)}
\begin{equation}\label{positive}
\tau_k(n+1)=\sum_{\substack{n=a_1^{f_1}\cdots a_r^{f_r};\\ f_i\leq k.}}(-1)^{f_1+\cdots f_r}{k\choose{f_1}}\cdots {k\choose{f_r}},
\end{equation}
\item{(b)}
\begin{equation}\label{negative}
\tau_{-k}(n+1)=\sum_{\substack{n=a_1^{f_1}\cdots a_r^{f_r}}}{f_1+k-1\choose{k-1}}\cdots {f_r+k-1\choose{k-1}}. 
\end{equation}
\end{description}
\end{theorem}
\begin{proof}
We have 
\begin{align*}
\sum_{n=1}^{\infty}\tau_k(n)q^n&=q\prod_{m=1}^{\infty}(1-q^m)^k\\ 
&=q\prod_{m=1}^{\infty}\left(1-{k\choose 1}q^{1\cdot m}+{k\choose 2}q^{2\cdot m}-\cdots +(-1)^{k}{k\choose k}q^{k\cdot m}\right).
\end{align*}
The above equality suggests that the value $(-1)^{f_1+f_2+\cdots +f_r}{k\choose {f_1}}{k\choose{f_2}}\cdots {k\choose{f_r}}$ contributes to the coefficient of $q^{n+1}$ for each partition of $n$ of the form $n=a_1^{f_1}\cdots a_r^{f_r}$ with the restriction $1\leq f_i\leq k$, and vice versa. Therefore, (a) is implied.

We have
\begin{align*}
\sum_{n=1}^{\infty}\tau_{-k}(n)q^n&=q\prod_{m=1}^{\infty}(1-q^m)^{-k}\\
&=q\prod_{m=1}^{\infty}\left({k-1\choose{k-1}}+{{(k-1)+1}\choose{k-1}}q^m+{(k-1)+2\choose{k-1}}q^{2m}+\cdots\right).
\end{align*}
The aforementioned equality suggests that the value ${{f_1+k-1}\choose{k-1}}\cdots {{f_r+k-1}\choose{k-1}}$ contributes to the coefficient of $q^{n+1}$ for each partition of $n$ of the type $n=a_1^{f_1}\cdots a_r^{f_r}$, and vice versa. Thus, (b) is implied.
\end{proof}
\subsection{ Parity Results connecting $\tau_k(n)$ and $R_t(n)$ at Specific Instances of $k$ and $t$}
We can get a parity result for $\tau_{k}(n)$ (which involves a partition function) using the partition sum representation mentioned in Theorem \ref{binomial-representation}.
\begin{definition}
Let $n$ be a positive integer and let $A$ be a set of positive integers. We define $F_A(n)$ to be the number of partitions of $n$ having each frequency from the set $A$.
\end{definition}
\begin{theorem}\label{parity-gen}
Let $k$ be a positive integer. Let $A=\{a\in\mathbb N:a\leq k, {{k}\choose{a}}\equiv 1\pmod{2}\}$. We have
\begin{equation}
\tau_k(n+1)\equiv F_A(n)\pmod{2}.
\end{equation}
\end{theorem}
\begin{proof}
From the representation given in (\ref{positive}) of Theorem \ref{binomial-representation}, the proof follows immediately.
\end{proof}
Using Theorem \ref{parity-gen}, we obtain a parity result for the $4$-regular partition function. 
\begin{theorem}\label{parityoftau}
Let $n$ be a positive integer. We have
\begin{equation}
R_4(n)\equiv\begin{cases}1\pmod{2}&{\text { if }n=\frac{m(m+1)}{2}};\\ 
0\pmod{2}&{\text{ otherwise}}.
\end{cases}
\end{equation}
\end{theorem}
\begin{proof}
We observe that ${24\choose k}\equiv 1\pmod 2$ if, and only if, $k\in\{8,16,24\}$. 
We may now write \begin{align*} \tau(n+1)&\equiv F_{\{8,16,24\}}(n)\pmod{2}, \end{align*} in accordance with Theorem \ref{parity-gen}.
As a result, $\tau(n+1)\equiv 0\pmod{2}$ for each $n\not\equiv 0\pmod{8}$. Therefore, the $n$ such that $n\equiv 0\pmod{8}$ is our primary concern.

The partitions of $n$ that the function $F_{\{8,16,24\}}(n)$ counts when $n\equiv 0\pmod{8}$ can be expressed as follows: \[n=a_1^8a_2^{16}a_3^{24}.\]
Alternatively expressed, \[\frac{n}{8}=a_1^1a_2^{2}a_3^{3}. \]

Consequently, $F_{\{8,16,24\}}(8m)$ counts the number of partitions of $m$ whose frequencies do not exceed $3$. We denote the number of such partitions by $d_3(m)$.

As can be seen,
\begin{align}
\notag \sum_{n=0}^{\infty}d_3(n)q^n&=(1+q+q^2+q^3)(1+q^2+q^4+q^6)(1+q^3+q^6+q^9)\cdots \\ 
&=\prod_{m=1}^{\infty}\frac{1-q^{4m}}{1-q^m}\label{secondline}\\
&\notag =\prod_{m=1}^{\infty}(1+q^m)(1+q^{2m}).
\end{align}

Considering that 
\[
\sum\limits_{n=0}^{\infty}q(n)q^n=\prod_{m=1}^{\infty}(1+q^m)
\] is the generating function for the number of partitions of $n$ with distinct parts (denoted $q(n)$), the equation 
\begin{equation}\label{d3}
d_3(n)=\sum_{s=0}^{\lfloor\frac{n}{2}\rfloor}q(n-2s)q(s) 
\end{equation}
is obtained from the above chain of equalities. 

We obtain \[q(s)\equiv \omega(s)\pmod 2,\] in light of Euler's Pentagonal Number theorem. Upon substituting this in Equation (\ref{d3}), we obtain
\begin{equation}\label{d3mod2}
d_3(n)\equiv \sum_{s=0}^{\lfloor\frac{n}{2}\rfloor}q(n-2s)\omega(s)\pmod 2.
\end{equation}
In view of (i) of Theorem 3 in \cite{david}, we have
\begin{equation}\label{d3final}
\sum_{s=0}^{\lfloor\frac{n}{2}\rfloor}q(n-2s)\omega(s)=\begin{cases}
1&\text{ if $\delta_t(n)=1$};\\ 
0&\text{ otherwise,}
\end{cases}
\end{equation}
where
\[\delta_t(n)=\begin{cases}
1&\text{ if $n=\frac{m(m+1)}{2}$;}\\
0&\text{ otherwise.}
\end{cases}
\]
At this point, we can see from Equation (\ref{secondline}) that $d_3(n)=R_4(n)$. Now the result follows from (\ref{d3mod2}) and (\ref{d3final}).
\end{proof}

\begin{remark}
The congruence $d_3(n)\equiv 1\pmod 2$ is true only when $n=\frac{m(m+1)}{2}$. This suggests that $\tau(n+1)$ is odd only if $\frac{n}{8}=\frac{m(m+1)}{2}$. Since $\frac{n}{8}=\frac{m(m+1)}{2}$ simplifies to $n+1=(2m+1)^2$, it follows that $\tau(n+1)$ is odd only if $n+1$ is an odd square. This is an established result. Ewell \cite{ewell-1} previously provided a proof.
\end{remark}
\begin{remark}
The number of $4$-regular partitions of $n$ is equal to the number of partitions of $n$ with frequencies from the set $\{1,2,3\}$, as we have concluded in the previous theorem's proof. This equinumerous statement can be generalized in the way that follows: the number of $(t+1)$-regular partitions of $n$ is equal to the number of partitions of $n$ with frequencies from the set $\{1,2,\cdots, t\}$. This generalization may be validated by the subsequent equalities. If one denotes the number of partitions of $n$ with frequencies not greater than $t$ by $d_t(n)$, then:
\begin{align*}
1+\sum_{n=1}^{\infty} d_t(n)q^n&=\prod_{m=1}^{\infty}(1+q^m+q^{2m}\cdots+ q^{tm})\\
&=\prod_{m=1}^{\infty}\frac{1-q^{(t+1)m}}{1-q^m}\\
&=1+\sum_{n=1}^{\infty}R_{t+1}(n)q^n.
\end{align*}
\end{remark}
\begin{remark}
 We find that $$\tau_{2k}(2n)\equiv 0\pmod{2}$$ without the use of Theorem \ref{parity-gen}. This is deduced from a general property of $f(q)\in\mathbb{Z}[[q]]$. Denote $f(q)=a_0+a_1q+a_2q^2+\cdots $. Then \begin{align*}
f(q)^2&=a_0a_0+(a_0a_1+a_0a_1)q+(a_0a_2+a_1a_1+a_2a_0)q^2+\cdots \\
&\equiv a_0^2+a_1^2 q^2+\cdots \pmod{2}\\
&\equiv a_0+a_1q^2+\cdots \pmod{2}.
\end{align*}
This gives 
\[f(q)^2\equiv f(q^2)\pmod{2}.
\]
Consequently, we have
\[qf(q)^{2k}\equiv qf(q^2)^{k}\pmod{2}.
\]
Now plugging $\prod\limits_{m=1}^{\infty}(1-q^m)$ in place of $f(q)$, we have $\tau_{2k}(2n)\equiv 0\pmod{2}$.
\end{remark}
It is not so simple to determine the parity of $\tau_{2k}(2n+1)$. We obtain parity expressions for $\tau_{14}(2n+1)$ and $\tau_{6}(2n+1)$ in the following result.
\begin{theorem}
Let $n$ be a positive integer. We have
\begin{description}
\item{(a)} $\tau_{14}(2n+1)\equiv R_8(n)\pmod{2}$;
\item{(b)} $\tau_6(2n+1)\equiv\begin{cases}1\pmod{2}&{\text { if }n=\frac{m(m+1)}{2}};\\ 0\pmod{2}&{\text{ otherwise}}.\end{cases}$
\end{description}
\end{theorem}
\begin{proof}
Define $A=\{1\leq a_i\leq 14:{14\choose{a_i}}\equiv 1\pmod{2}\}=\{2,4,6,8,10,12,14\}$.
Considering Theorem \ref{parity-gen}, we can now write
\begin{equation*}\label{tau14mod2}
\tau_{14}(n+1)\equiv F_{\{2,4,6,8,10,12,14\}}(n)\pmod{2}.
\end{equation*}
All partitions of $n$ that $F_{\{2,4,6,8,10,12,14\}}(n)$ counts while $n$ is even are of the type
\[n=a_1^2a_2^4a_3^6a_4^8a_5^{10}a_6^{12}a_7^{14}.
\]
Alternatively expressed,
\[\frac{n}{2}=a_1^1a_2^2a_3^3a_4^4a_5^5a_6^6a_7^7.
\]
As a result, $F_{\{2,4,6,8,10,12,14\}}(n)$ counts the number of $\frac{n}{2}$ partitions with frequencies that do not exceed $7$. To put it another way, $F_{\{2,4,6,8,10,12,14\}}(n)=d_7(\frac{n}{2})=R_8(\frac{n}{2})$. Thus, considering the even $n$, we have
\begin{align*}
\tau_{14}(n+1)& \equiv F_{\{2,4,6,8,10,12,14\}}(n)\pmod{2}\\
& \equiv R_8\left(\frac{n}{2}\right)\pmod{2}.
\end{align*}
Now (a) follows.
A similar search together with the parity result of $R_4(n)$ (mentioned in Theorem \ref{parityoftau}) gives (b).
\end{proof}

The following parity result relates $R_{2^s}(n)$ and $\tau_{2^s-1}(n)$. This follows from a parity result concerning binomial coefficients.
\begin{theorem}
For every positive integer $s$, we have
\begin{equation}
R_{2^s}(n)\equiv\tau_{2^s-1}(n+1)\pmod 2.
\end{equation}
\end{theorem}
\begin{proof}
We note that $\frac{2^s-2}{2}=2^{s-1}-1$ is the largest integer that does not exceed $\frac{2^s-1}{2}$. Stated in another way, 
\[
\lfloor\frac{2^s-1}{2}\rfloor=2^{s-1}-1. 
\]
Using the result of James Glaisher \cite{james}, we obtain
\begin{equation}\label{glaisher}
{n\choose{k}}\equiv\begin{cases}
0\pmod{2}&\text{ if $n$ is even and $k$ is odd;}\\
\binom{\lfloor\frac{n}{2}\rfloor}{\lfloor\frac{k}{2}\rfloor}\pmod{2}&\text{ otherwise.}
\end{cases}
\end{equation}
We obtain the following congruence relation by substituting $2^s-1$ for $n$ in (\ref{glaisher}): 
\[{{2^s-1}\choose{k}}\equiv{{2^{s-1}-1}\choose{\lfloor\frac{k}{2}\rfloor}}\pmod{2}.
\]
After using the previously specified modulo 2 reduction $s-1$ times, taking the right-side term for subsequent reduction, we obtain
\[{{2^s-1}\choose{k}}\equiv1\pmod{2}.
\]
Now that the aforementioned observation has been made, we may write
\begin{align*}
\sum_{n=0}^{\infty}\tau_{2^s-1}(n+1)q^n&=\prod_{m=1}^{\infty}(1-q^m)^{2^s-1}\\
&\equiv\prod_{m=1}^{\infty}(1+q^m+\cdots +q^{2^{{(s-1)}m}})\pmod{2}\\ 
&\equiv\prod_{m=1}^{\infty}\frac{1-q^{2^sm}}{1-q^m}\pmod{2}\\
&\equiv\sum_{n=0}^{\infty}R_{2^s}(n)q^n\pmod{2}.
\end{align*}
The proof is now completed.
\end{proof}

\subsection{ Ramanujan's Tau Function Modulo $3$, $5$, $7$, $11$, $13$, $17$, $23$ and $25$}
This section is concerned with deriving a simple expression for $\tau(n)$ modulo $m$ when $m\in\{3,5,7,11,13,17,23,25\}$. The derivations of this section just rely on some arithmetic properties of ${24\choose{s}}$.
\begin{theorem}\label{taumodulo3}
Let $n$ be a positive integer. We have
\begin{equation}
\tau(n+1)\equiv\begin{cases}
R_9\left(\frac{n}{3}\right)\pmod 3&\text{ if $3\mid n$};\\ 
0\pmod 3&\text{ otherwise}.
\end{cases}
\end{equation}
\end{theorem}
\begin{proof}
Given the following observations:
\begin{enumerate}
\item
${24\choose{k}}\equiv0\pmod{3}$ 
when $k\in\{1,2,4,5,7,8,10,11,13,14,16,17,19,20,22,23\}$,
\item ${24\choose{k}}\equiv-1\pmod{3}$
when $k\in\{3,9,15,21\}$,
\item ${24\choose{k}}\equiv1\pmod{3}$ when $k\in\{6,12,18,24\}$,
\end{enumerate}
we may write
\begin{align*}
\sum_{n=0}^{\infty}\tau(n+1)q^n
&\equiv\prod_{m=1}^{\infty}(1+q^{3m}+q^{6m}+\cdots +q^{24m})\pmod{3}\\
&\equiv\prod_{m=1}^{\infty}\frac{1-q^{9\times 3m}}{1-q^{3m}}\pmod{3}.
\end{align*}
Since 
\[\prod_{m=1}^{\infty}\frac{1-q^{9m}}{1-q^m}=\sum_{n=0}^{\infty}R_9(n)q^n,
\]
in view of the above observation, we obtain the following congruence:
\begin{equation*}
\tau(n+1)\equiv\begin{cases}
R_9\left(\frac{n}{3}\right)\pmod 3&\text{ if $3\mid n$},\\ 
0\pmod 3&\text{ otherwise}.
\end{cases}
\end{equation*}
\end{proof}
As an immediate consequence of the theorem above, we obtain the following result of Ramanujan.
\begin{corollary}[Ramanujan \cite{ramanujan-20}]
Let $n$ be a positive integer. We have
\begin{equation*}
\tau(3n)\equiv 0\pmod{3}.
\end{equation*}
\end{corollary}
\begin{proof}
Theorem \ref{taumodulo3} allows us to write
\begin{align*}
\tau(3n)&=\tau(3n-1+1)\\
&\equiv 0\pmod{3}.
\end{align*}
\end{proof}
As another consequence of Theorem \ref{taumodulo3}, we obtain the following expression for $R_9(n)$ modulo $3$. 
\begin{corollary}
Let $n$ be a positive integer. We have
\begin{equation}
R_9(n)\equiv \sigma(3n+1)\pmod{3}.
\end{equation}
\end{corollary}
\begin{proof}
We have 
\[\tau(n)\equiv n\sigma(n)\pmod{3}
\]
from the works of Ramanujan \cite[p. 112]{george}. 
It follows therefrom that
\[\tau(n)\equiv\begin{cases}
0\pmod{3}&\text{ if }3\mid n;\\
\sigma(n)\pmod{3}&\text{ if }\gcd(n,3)=1.
\end{cases}
\]
Given the above observation, Theorem \ref{taumodulo3} allows us to write
\begin{align*}
R_9(n)&\equiv \tau(3n+1)\pmod{3}\\
&\equiv \sigma(3n+1)\pmod{3}.
\end{align*}
\end{proof}
\begin{theorem}\label{taumodulo5}
Let $n$ be a positive integer. We have
\begin{equation}
\tau(n+1)\equiv R_{25}(n)\pmod{5}.
\end{equation}
\end{theorem}
\begin{proof}
We observe that
\begin{enumerate}
\item ${24\choose{k}}\equiv1\pmod{5}$ when $k\in\{0,2,4,6,8,10,12,14,16,18,20,22,24\}$,
\item ${24\choose{k}}\equiv-1\pmod{5}$ when $k\in\{1,3,5,7,9,11,13,15,17,19,21,23\}$.
\end{enumerate}
Based on these observations, we can write
\begin{align*}
\sum_{n=0}^{\infty}\tau(n+1)q^{n}
&\equiv\prod_{m=1}^{\infty}(1+q^m+q^{2m}+\cdots +q^{24m})\pmod{5}\\
&\equiv\prod_{m=1}^{\infty}\frac{1-q^{25m}}{1-q^m}\pmod{5}.
\end{align*}
Since 
\[\prod_{m=1}^{\infty}\frac{1-q^{25m}}{1-q^m}=\sum_{n=0}^{\infty}R_{25}(n)q^n,
\]
we obtain from the above observation that
\[\tau(n+1)\equiv R_{25}(n)\pmod{5}.
\]
\end{proof}
An expression for $R_{25}$ modulo $5$ can be obtained by applying the aforementioned theorem.
\begin{corollary}
Let $n$ be a positive integer. We have
\begin{equation}
R_{25}(n)\equiv (n+1)\sigma(n+1)\pmod{5}.
\end{equation}
\end{corollary}
\begin{proof}
Wilton \cite{wilton} established that 
\begin{equation*}
\tau(n)\equiv n\sigma(n)\pmod{5}.
\end{equation*}
We may now write in light of Theorem \ref{taumodulo5}:
\begin{align*}
R_{25}(n)&\equiv \tau(n+1)\pmod{5}\\
&\equiv (n+1)\sigma(n+1)\pmod{5}.
\end{align*}
\end{proof}
\begin{theorem}\label{taumodulo7}
Let $n$ be a positive integer. We have
\begin{equation}
\tau(n+1)\equiv \sum_{n=\frac{m(m+1)}{2}+7\frac{r(r+1)}{2}}(-1)^{m+r}(2m+1)(2r+1)\pmod{7}.
\end{equation}
\end{theorem}
\begin{proof}
We observe that
\[
{24\choose{k}}\equiv \begin{cases}
1\pmod{7}&\text{ if $k$=0, 3, 21, 24;}\\ 
0\pmod{7}&\text{ if $k$=4, 5, 6, 11, 12, 13, 18, 19, 20;}\\
3\pmod{7}&\text{ if $k$=1, 2, 22, 23;}\\
-4\pmod{7}&\text{ if $k$=7, 10, 14, 17;}\\
-12\pmod{7}&\text{ if $k$=8, 9, 25, 16.}\\
\end{cases}
\]
Based on this observation, we can write 
\begin{align*}
\sum_{n=0}^{\infty}\tau(n+1)q^n
&\equiv\prod_{m=1}^{\infty} \left[\left(1-3q^m+3q^{2m}-q^{3m}\right)+4\left(q^{7m}-3q^{8m}+3q^{9m}-q^{10m}\right)\right.\\
&\text{\ \ \ \ \ \ \ \ }-4\left(q^{14m}-3q^{15m}+3q^{16m}-q^{17m}\right)\\
&\text{\ \ \ \ \ \ \ \ }\left.-\left(q^{21m}-3q^{22m}+3q^{23m}-q^{24m}\right)\right]\pmod{7}
\end{align*}
\begin{align*}
\ \ \ \ \ \ \ \ \ \ \ \ \ \ \ \ \ \ \ \ &\equiv\prod_{m=1}^{\infty}\left[\left(1-3q^m+3q^{2m}-q^{3m}\right)-3\left(q^{7m}-3q^{8m}+3q^{9m}-q^{10m}\right)\right.\\
&\text{\ \ \ \ \ \ \ \ }+3\left(q^{14m}-3q^{15m}+3q^{16m}-q^{17m}\right)\\
&\text{\ \ \ \ \ \ \ \ }\left.-\left(q^{21m}-3q^{22m}+3q^{23m}-q^{24m}\right)\right]\pmod{7}\\
&\equiv\prod_{m=1}^{\infty}\left[\left(1-3q^m+3q^{2m}-q^{3m}\right)\left(1-3q^{7m}+3q^{14m}-q^{21m}\right)\right]\pmod{7}\\ 
&\equiv\prod_{m=1}^{\infty}\left(1-q^{m}\right) ^3 \prod_{r=1}^{\infty}\left(1-q^{7r}\right)^3\pmod{7}.
\end{align*}
Jacobi's triple product identity states that
\begin{equation}\label{JTPI}
\prod_{n=1}^{\infty}(1-q^n)^3=\sum_{s=0}^{\infty}a_sq^s,
\end{equation}
where 
\[a_s=\begin{cases}
(-1)^t(2t+1)&\text{ if $s=\frac{t(t+1)}{2}$;}\\
0&\text{ otherwise.}
\end{cases}
\]
The intended congruence will result from applying this identity to the tail end product of the preceding chain of expressions.
\end{proof}
\begin{theorem}
Let $n$ be a positive integer. We have
\begin{equation}
\tau(n+1)\equiv\sum_{n=\frac{3l^2\pm l}{2}+\frac{3m^2\pm m}{2}+11\frac{3s^2\pm s}{2}+11\frac{3r^2\pm r}{2}}(-1)^{l+m+s+r}\pmod{11}.
\end{equation}
\end{theorem}
\begin{proof}
Since 
\[{24\choose{k}}\equiv\begin{cases}
1\pmod{11}&\text{ when $k$=0, 2, 22, 24;}\\
2\pmod{11}&\text{ when $k$=1, 11, 13, 23;}\\
4\pmod{11}&\text{ when $k$=12;}\\
0\pmod{11}&\text{ when $k$=3, 4, 5, 6, 7, 8, 9, 10, 14, 15, 16, 17, 18, 19, 20, 21,}
\end{cases}
\]
we can write
\begin{align*}
\sum_{n=0}^{\infty}\tau(n+1)q^n
&\equiv\prod_{m=1}^{\infty}\left[\left(1-2q^m+q^{2m}\right)-2\left(q^{11m}-2q^{12m}+q^{13m}\right)\right.\\
&\text{\ \ \ \ \ \ \ \ }+\left.\left(q^{22m}-2q^{23m}+q^{24m}\right)\right]\pmod{11}\\
&\equiv \prod_{m=1}^{\infty}\left[(1-q^m)^2(1-q^{11m})^2\right]\pmod{11}.
\end{align*}
Applying Euler's pentagonal number theorem now results in the following equality:
\[\prod_{m=1}^{\infty}\left[(1-q^m)^2(1-q^{11m})^2\right]=1+\sum_{n=\frac{3l^2\pm l}{2}+\frac{3m^2\pm m}{2}+11\frac{3s^2\pm s}{2}+11\frac{3r^2\pm r}{2}}(-1)^{l+m+s+r}q^n.
\]
This insight will yield the expected congruence when applied to the tail end product of the aforementioned chain of expressions. 
\end{proof}
\begin{theorem}
Let $n$ be a positive integer. We have
\begin{equation}
\tau(n+1)\equiv \sum_{n=13\times\frac{3r^2\pm r}{2}+s}(-1)^r\tau_{11}(s)\pmod{13}.
\end{equation}
\end{theorem}
\begin{proof}
We observe that 
\[{24\choose{k}}\equiv\begin{cases}
{11\choose{k}}\pmod{13}&\text{ when $0\leq k\leq 11$;}\\
0\pmod{13}&\text{ when $k=12$;}\\ 
{11\choose{24-k}}\pmod{13}&\text{ when $13\leq k\leq 24$.}
\end{cases}
\]
This observation enables us to write
\begin{align*}
\sum_{n=0}^{\infty}\tau(n+1)q^n
&\equiv\prod_{m=1}^{\infty}\left[\left({11\choose{0}}-{11\choose{1}}q^m+\cdots -{11\choose{11}}q^{11m}\right)\right.\\
&\text{\ \ \ \ \ \ \ \ }-q^{13m}\left.\left({11\choose{0}}-{11\choose{1}}q^m+\cdots -{11\choose{11}}q^{11m}\right)\right]\pmod{13}\\
&\equiv \prod_{m=1}^{\infty}\left[(1-q^m)^{11}(1-q^{13m})\right]\pmod{13}.
\end{align*}
Now given the definition of $\tau_{11}(n)$ and Euler's pentagonal number theorem, we may write 
\[\prod_{m=1}^{\infty}(1-q^m)^{11}(1-q^{13m})=1+\left(\sum_{n=13\times\frac{3r^2\pm r}{2}+s}(-1)^r\tau_{11}(s)\right)q^n.
\]
This observation will yield the expected congruence when applied to the tail end product of the aforementioned chain of expressions.
\end{proof}
\begin{theorem}
Let $n$ be a positive integer. We have
\begin{equation}
\tau(n+1)\equiv\sum_{n=17\times\frac{3r^2\pm r}{2}+s}(-1)^r\tau_7(s)\pmod{17}.
\end{equation}
\end{theorem}
\begin{proof}
Since
\[{24\choose{k}}\equiv\begin{cases}
{7\choose{k}}\pmod{17}&\text{ when $0\leq k\leq 7$;}\\ 
{7\choose{24-s}}\pmod{17}&\text{ when $0\leq 24-s\leq 7$;}\\ 
0\pmod{17}&\text{ otherwise,}
\end{cases}
\]
one can write
\begin{align*}
\sum_{n=0}^{\infty}\tau(n+1)q^n
&\equiv\prod_{m=1}^{\infty}\left[\left({7\choose{0}}-{7\choose{1}}q^m+\cdots -{7\choose{7}}q^{7m}\right)\right.\\
&\text{\ \ \ \ \ \ \ \ }-q^{17m}\left.\left({7\choose{0}}-{7\choose{1}}q^m+\cdots -{7\choose{7}}q^{7m}\right)\right]\pmod{17}\\
&\equiv \prod_{m=1}^{\infty}(1-q^m)^{7}(1-q^{17m})\pmod{17}.
\end{align*}
One can have
\[\prod_{m=1}^{\infty}(1-q^m)^{7}(1-q^{17m})=1+\left(\sum_{n=17\times\frac{3r^2\pm r}{2}+s}(-1)^r\tau_7(s)\right)q^n
\]
based on the definition of $\tau_7(n)$ and Euler's pentagonal number theorem. This observation will yield the expected congruence when applied to the tail end product of the aforementioned chain of expressions.
\end{proof}
\begin{theorem}
Let $n$ be a positive integer. We have
\begin{equation}
\tau(n+1)\equiv\sum_{n=19\times\frac{3r^2\pm r}{2}+s}(-1)^r\tau_5(s)\pmod{19}.
\end{equation}
\end{theorem}
\begin{proof}
Since
\[{24\choose{k}}\equiv\begin{cases}
1\pmod{19}&\text{ when $k=0, 5, 19, 24$;}\\
5\pmod{19}&\text{ when $k=1 ,4, 20, 23$;}\\
10\pmod{19}&\text{ when $k=2, 3, 21, 22$;}\\ 
0\pmod{19}&\text{ otherwise,}
\end{cases}
\]
one can write
\begin{align*}
\sum_{n=0}^{\infty}\tau(n+1)q^n
&\equiv\prod_{m=1}^{\infty}\left[\left(1-5q^m+10q^{2m}-10q^{3m}+5q^{4m}-q^{5m}\right)\right.\\
&\text{\ \ \ \ \ \ \ \ }-q^{19m}\left.\left(1-5q^m+10q^{2m}-10q^{3m}+5q^{4m}-q^{5m}\right)\right]\pmod{19}\\
&\equiv \prod_{m=1}^{\infty}\left[(1-q^m)^{5}(1-q^{19m})\right]\pmod{19}.
\end{align*}
Now from the definition of $\tau_5(n)$ and Euler's pentagonal number theorem, we have
\[\prod_{m=1}^{\infty}(1-q^m)^{5}(1-q^{19m})=1+\left(\sum_{n=19\times\frac{3r^2\pm r}{2}+s}(-1)^r\tau_5(s)\right)q^n.
\]
While applying this observation in the tail end product of the above chain of expressions, we will get the expected congruence.
\end{proof}
Utilizing the following Ramanujan's formula \cite[pp. 163-164]{hardy} for $\tau(p^r)$:
\[\tau(p^r)=\frac{p^{\frac{11}{2}r}}{\sin{\psi_p}}\sin{(r+1)\psi_p},
\] 
where $p$ is a prime number and $\cos{\psi_p}=\frac{\tau(p)}{2p^{\frac{11}{2}}}$, Lehmer \cite{Lehmer} gave an expression for $\tau(n)\mod{23}$:
\begin{equation}\label{lehmer}
\tau(n)\equiv\sigma_{11}(n_1)2^t3^{\frac{-t}{2}}\prod\limits_{i=1}^{t}\sin{\frac{2\pi}{3}(1+\alpha_i)}\pmod{23},
\end{equation}
where $n=n_1\prod\limits_{i=1}^{t}p_i^{\alpha_i}$, $p_i$s are the only prime factors of $n$ which are not of the form $u^2+23v^2$ but are quadratic residues of 23, and $\alpha_i$ is the exponent of the highest power of $p_i$ dividing $n$. 

We provide an expression for $\tau(n)$ modulo $23$ in the following result, which is quite simple in comparison to (\ref{lehmer}).
\begin{theorem}
Let $n$ be a positive integer. We have
\begin{equation}
\tau(n+1)\equiv\sum_{n=\frac{3r^2\pm r}{2}+23\times\frac{3s^2\pm s}{2}}(-1)^{r+s}\pmod{23}.
\end{equation}
\end{theorem}
\begin{proof}
Since 
\[{24\choose{k}}\equiv\begin{cases}
1\pmod{23}&\text{ if $k=0, 1, 23, 24$;}\\
0\pmod{23}&\text{ otherwise,}
\end{cases}
\]
we can write
\begin{align*}
\sum_{n=0}^{\infty}\tau(n+1)q^n
&\equiv\prod_{m=1}^{\infty}\left(1-q^m-q^{23m}+q^{24m}\right)\pmod{23}\\
&\equiv\prod_{m=1}^{\infty}\left[(1-q^m)(1-q^{23m})\right]\pmod{23}.
\end{align*}
Write 
\[\prod_{m=1}^{\infty}\left[(1-q^m)(1-q^{23m})\right]=\sum_{n=0}^{\infty}a_nq^n.
\]
Now in accordance with Euler's pentagonal number theorem, we have
\[a_n=\begin{cases}
(-1)^{r+s}&\text{ if $n=\frac{3r^2\pm r}{2}+23\times\frac{3s^2\pm s}{2}$;}\\
0&\text{ otherwise.}
\end{cases}
\]
While applying this observation in the tail end product of the above chain of expressions, we will get the intended congruence.
\end{proof}
\begin{theorem}
Let $n$ be a positive integer. We have
\begin{equation}
\tau(n+1)\equiv \sum_{n=r+5\frac{s(s+1)}{2}+5\frac{3t^2\pm t}{2}}(-1)^{s+t}(2s+1)R_{5}(r)\pmod{25}.
\end{equation}
\end{theorem}
\begin{proof}
Since 
\[{24\choose{k}}\equiv\begin{cases}
1&\text{ when $k$=0, 2, 4, 20, 22, 24;}\\
-1&\text{ when $k$=1, 3, 21, 23;}\\
4&\text{ when $k$=5, 7, 9, 15, 17, 19;}\\
-4&\text{ when $k$=6, 8, 16, 18;}\\
6&\text{ when $k$=10, 12, 14;}\\
-6&\text{ when $k$=11, 13,}
\end{cases}
\]
we can write
\begin{align*}
\sum_{n=0}^{\infty}\tau(n+1)q^n
&\equiv\prod_{m=1}^{\infty}\left[\left(1+q^m+q^{2m}+q^{3m}+q^{4m}\right)\right.\\ &\text{\ \ \ \ \ \ \ \ }-4\left(q^{5m}+q^{6m}+q^{7m}+q^{8m}+q^{9m}\right)\\
&\text{\ \ \ \ \ \ \ \ }+6\left(q^{10m}+q^{11m}+q^{12m}+q^{13m}+q^{14m}\right)\\
&\text{\ \ \ \ \ \ \ \ }-4\left(q^{15m}+q^{16m}+q^{17m}+q^{18m}+q^{19m}\right)\\
&\text{\ \ \ \ \ \ \ \ }+\left.\left(q^{20m}+q^{21m}+q^{22m}+q^{23m}+q^{24m}\right)\right]\pmod{25}
\end{align*}
\begin{align*}
&\equiv\prod_{m=1}^{\infty}\left[\frac{1-q^{5m}}{1-q^m}-4q^{5m}\frac{1-q^{5m}}{1-q^m}+6q^{10m}\frac{1-q^{5m}}{1-q^m}-4q^{15m}\frac{1-q^{5m}}{1-q^m}\right.\\
&\text{\ \ \ \ \ \ \ \ }\left.+q^{20m}\frac{1-q^{5m}}{1-q^m}\right]\pmod{25}\\
&\equiv\prod_{m=1}^{\infty}\frac{1-q^{5m}}{1-q^m}\prod_{r=1}^{\infty}(1-q^{5r})^4\pmod{25}.
\end{align*}
Given the generating function of $R_5(n)$, Euler's pentagonal number theorem and Jacobi's triple product identity, we may write
\[\prod_{m=1}^{\infty}\frac{1-q^{5m}}{1-q^m}\prod_{r=1}^{\infty}(1-q^{5r})^4=1+\sum_{\substack{n\in\mathbb N\\ n=r+5\frac{s(s+1)}{2}+5\frac{3t^2\pm t}{2}}}(-1)^{s+t}(2s+1)R_{5}(r)q^n.
\]
While applying this observation in the tail end product of the above chain of expressions, we will get the intended congruence.
\end{proof}
\subsection{ Prime Moduli}
Let $p$ be a prime number. This section provides an expression for $\tau_k(n)$ modulo $p$ when $k\in\{p^s:s\in\mathbb N\}\cup\{2p, 2p+1, p^2+1\}$.
\begin{theorem}\label{tau-prime-modulo}
Let $p$ be a prime number, and let $s$ be a positive integer. We have
\begin{equation}
\tau_{p^s}(n+1)\equiv\begin{cases}
(-1)^t\pmod{p}&\text{ if $n=\frac{p^s(3t^2\pm t)}{2}$};\\ 
0\pmod{p}&\text{ otherwise.}
\end{cases}
\end{equation}
\end{theorem}
\begin{proof}
Since 
\[{{p^s}\choose{t}}\equiv0\pmod{p}
\]
for every $t\in\{1,2,\ldots ,p^s-1\}$, we can write
\begin{align*}
\sum_{n=1}^{\infty}\tau_{p^s}(n)q^{n-1}&=\prod_{m=1}^{\infty}(1-q^m)^{p^s}\\
&\equiv\prod_{m=1}^{\infty}(1-q^{p^sm})\pmod{p}.
\end{align*}
Using Euler's pentagonal number theorem, we now obtain
\[\prod_{m=1}^{\infty}(1-q^{p^sm})=\sum_{r=0}^{\infty}\omega_{p^s}(r)q^r,
\]
where 
\[\omega_{p^s}(r)=\begin{cases}
(-1)^t&\text{ if $\frac{r}{p^s}=\frac{(3t^2\pm t)}{2}$};\\ 
0&\text{ otherwise}.
\end{cases}
\]
While applying this observation in the tail end product of the above chain of expressions, we will get the intended congruence.
\end{proof}
\begin{definition}
Let $m$ be a positive integer. Let $\vartheta_p(m)$ be defined as a positive integer $k$ such that $p^k\mid m$ but $p^{k+1}\nmid m$.
\end{definition}
\begin{theorem}Let $p$ be an odd prime number. We have
\begin{equation}
\tau_{2p}(n+1)\equiv\sum_{n+1=p\left(\frac{3r^2\pm r}{2}+\frac{3s^2\pm s}{2}\right)}(-1)^{r+s}\pmod{p}.
\end{equation}
\end{theorem}
\begin{proof}
Clearly ${2p\choose{0}}\equiv 1\pmod{p}$ and ${2p\choose{2p}}\equiv1\pmod{p}$. We observe that $\vartheta_p(2p\times (2p-1)\times \cdots \times (2p-(k-1))=1$ and $\vartheta_p(1\times 2\times \cdots \times k)=0$ when $1\leq k\leq p-1$. Consequently, ${{2p}\choose{k}}\equiv 0\pmod{p}$ when $1\leq k \leq p-1$. Also, we have
\[{{2p}\choose{p}}=\frac{2p\times (2p-1)\times\cdots 
\times(p+1)}{1\times 2\times \cdots \times p}=\frac{2\times (2p-1)\times\cdots 
\times(p+1)}{(p-1)!}.
\]
Since
$2p-1\equiv-1\pmod{p}$, $2p-2\equiv -2\pmod{p}$, $\cdots$, $p+1\equiv -(p-1)\pmod{p}$, we obtain (in light of Wilson's theorem) that 
\begin{align*}
2\times(2p-1)\times\cdots 
\times(p+1)&\equiv2(p-1)!\pmod{p}\\
&\equiv -2\pmod{p}.
\end{align*} 
Consequently, ${{2p}\choose{p}}$ is of the form $\frac{rp-2}{kp-1}$. Given this form, we obtain ${{2p}\choose{p}}-2=\frac{rp-2}{kp-1}-2=\frac{(r-2k)p}{kp-1}$.
Therefrom, it follows that ${{2p}\choose{p}}\equiv 2\pmod{p}$. Moreover, since ${{2p}\choose{k}}={{2p}\choose{2p-k}}$, we can write
\begin{align*}
\sum_{n=0}^{\infty}\tau_{2p}(n+1)q^n
&\equiv\prod_{m=1}^{\infty}\left(1+(-1)^{p}2q^{pm}+(-1)^{2p}q^{2pm}\right)\pmod{p}\\ 
&\equiv\prod_{m=1}^{\infty}(1-q^{pm})^2\pmod{p}.
\end{align*}
Since 
\[\prod_{m=1}^{\infty}(1-q^{pm})^2=1+\left(\sum_{n+1=p\left(\frac{3r^2\pm r}{2}+\frac{3s^2\pm s}{2}\right)}(-1)^{r+s}\right)q^n,
\]
the result follows as a consequence of the above observation.
\end{proof}
\begin{corollary}
Let $p$ be an odd prime number. If $p\nmid n+1$ then
\[\tau_{2p}(n+1)\equiv 0\pmod{p}.
\]
\end{corollary}
\begin{theorem}
Let $p$ be an odd prime number. Then we have
\begin{equation}
\tau_{2p+1}(n+1)\equiv\sum_{n+1=\frac{3r^2\pm r}{2}+p\left(\frac{3s^2\pm s}{2}+\frac{3t^2\pm t}{2}\right)}(-1)^{r+s+t}\pmod{p}.
\end{equation}
\end{theorem}
\begin{proof}
In general, the congruences ${{2p+1}\choose{0}}\equiv {{2p+1}\choose{1}}\equiv 1\pmod{p}$ are true.
Consider the product:
\begin{equation}\label{product}
\prod_{k=1}^{2p+1}\frac{2p+1-(k-1)}{k}.
\end{equation}
Here, the product of the first $r$ terms gives the value ${{2p+1}\choose{r}}$. Based on this observation, we deduce that ${{2p+1}\choose{r}}\equiv 0\pmod{p}$, for $r$ limited to the bound $2\leq r\leq p-1$.

We will now show that ${{2p+1}\choose{p}}\equiv 2\pmod{p}$. To that end, consider the product of the first $p$ terms of (\ref{product}):
\[\frac{2p+1}{1}\cdot\frac{2p}{2}\cdot\frac{2p-1}{3}\cdots\frac{(2p+1)-(p-1)}{p}.
\]
After cancelling $p$ in the above product, we obtain the following term:
\[\frac{(2p+1)\times{2}\times(2p-1)\times\cdots\times{((2p+1)-(p-1))}}{(p-1)!}.
\]
Wilson's theorem allows us to write
\begin{align*}
(2p+1)\times{2}\times(2p-1)\times\cdots\times((2p+1)-(p-1))&\equiv -2\times(p-2)!\pmod{p}\\
&\equiv -2\pmod{p}.
\end{align*}
The form of ${{2p+1}\choose{p}}$ is thus $\frac{sp-2}{rp-1}$. From this, we have ${{2p+1}\choose{p}}-2=\frac{sp-2}{rp-1}-2=\frac{(s-2r)p}{rp-1}\equiv 0\pmod{p}$.

Considering that ${{2p+1}\choose{k}}={{2p+1}\choose{2p+1-k}}$, we can write
\begin{align*}
\sum_{n=0}^{\infty}\tau_{2p+1}(n+1)q^n
&\equiv\prod_{m=1}^{\infty}\left(1-q^m-2q^{pm}+2q^{(p+1)m}\right.\\
&\text{\ \ \ \ \ \ \ \ \ \ \ \ }\left.+q^{2pm}-q^{(2p+1)m}\right)\pmod{p}\\ 
&\equiv\prod_{m=1}^{\infty}\left[(1-q^m){(1-q^{pm})^2}\right]\pmod{p}.
\end{align*}
In light of Euler's pentagonal number theorem, we obtain
\[
\prod_{m=1}^{\infty}(1-q^m)(1-q^{pm})^2=1+\left(\sum_{n=\frac{3r^2\pm r}{2}+p\left(\frac{3s^2\pm s}{2}+\frac{3t^2\pm t}{2}\right)}(-1)^{r+s+t}\right)q^n.
\]
By using this in the congruence mentioned above, we obtain the intended congruence. 
\end{proof}
\begin{theorem}
Let $p$ be an odd prime number. We have
\begin{equation}
\tau_{p^2+1}(n+1)\equiv\sum_{n+1=\frac{3r^2\pm r}{2}+p^2\frac{3s^2\pm s}{2}}(-1)^{r+s}\pmod{p}.
\end{equation}
\end{theorem}
\begin{proof}
The following congruences are all evident: ${{p^2+1}\choose{0}}\equiv 1\pmod{p}$, ${{p^2+1}\choose{1}}\equiv 1\pmod{p}$, ${{p^2+1}\choose{p^2}}\equiv 1\pmod{p}$ and ${{p^2+1}\choose{p^2+1}}\equiv 1\pmod{p}$. 

Consider the term 
\[
{{p^2+1}\choose{k}}=\frac{(p^2+1)\times p^2\times\cdots \times (p^2+1-(k-1))}{1\times 2\times\cdots\times k}.
\]

Assume $1\leq a\leq p-1$. It can then be observed that, for $(a-1)p+1\leq k-1\leq ap$, $$\vartheta_p((p^2+1)\times (p^2+1-1)\times\cdots \times(p^2+1-(k-1)))=a+1$$ and $$\vartheta_p(1\times 2\times \cdots \times k)=a-1 \text{ or } a.$$

Assume $k-1\in\{(p-1)p+1,\cdots ,p^2-2\}$. It is easy to see that $$\vartheta_p((p^2+1)\times p^2\times\cdots \times(p^2-(k-1)))=p+1$$
and 
$$\vartheta_p(1\times 2\times \cdots \times k)=p-1.$$

Consequently, ${{p^2+1}\choose{k}}\equiv 0\pmod{p}$ when $2\leq k\leq p^2-1$. Based on this observation, we may write
\begin{align*}
\sum_{n=0}^{\infty}\tau_{p^2+1}(n+1)q^n
&\equiv\prod_{m=1}^{\infty}\left(1-q^m+(-1)^{p^2}q^{p^2m}+(-1)^{p^2+1}q^{(p^2+1)m}\right)\pmod{p}\\ 
&\equiv\prod_{m=1}^{\infty}(1-q^m)(1-q^{p^2m})\pmod{p}.
\end{align*}
Since
\[\prod_{m=1}^{\infty}(1-q^m)(1-q^{p^2m})=1+\left(\sum_{n=\frac{3r^2\pm r}{2}+p^2\frac{3s^2\pm s}{2}}(-1)^{r+s}\right)q^n,
\]
the result follows.
\end{proof}
\subsection{ Congruence Properties of $R_9(n)$ modulo ${3}$ and $R_p(n)$ modulo ${p}$, where $p$ is a Prime Number}
In this section, we apply Ewell's congruence for $\tau(n)$ to obtain a recursive congruence relation for $R_9(n)$ modulo $3$. Additionally, for any prime number $p$, we derive an expression for $R_p(n)$ modulo ${p}$.
\begin{theorem}\label{nine-regular}
Let $n$ be a positive integer. We have
\begin{equation}\label{R9-rec}
R_9(4n+1)\equiv R_9(n)\pmod{3}.
\end{equation}
\end{theorem}
\begin{proof}
Theorem \ref{taumodulo3} allows us to write
$$
R_9(n)\equiv \tau(3n+1)\pmod 3.
$$
Ewell \cite{ewell} proved that 
$$
\tau(4n)\equiv\tau(n)\pmod{3}.
$$
Given these insights, we may write
\begin{align*}
R_9(4n+1)&\equiv\tau(4(3n+1))\pmod{3}\\
&\equiv\tau(3n+1)\pmod{3}\\
&\equiv R_9(n)\pmod{3}.
\end{align*}
Now the result follows.
\end{proof}
\begin{corollary} Let $r$ and $s$ be positive integers. We have
\begin{equation}\label{r9}
R_9\left((r-1)4^{s-1}+\frac{4^{s}-1}{3}\right)\equiv R_9(r)\pmod{3}.
\end{equation}
In particular,
\[R_9\left(\frac{4^s-1}{3}\right)\equiv1\pmod{3},
\]
\[R_9\left(4^{s-1}+\frac{4^s-1}{3}\right)\equiv2\pmod{3}
\]
and
\[R_9\left(2\times 4^{s-1}+\frac{4^s-1}{3}\right)\equiv0\pmod{3}.
\]
\end{corollary}
\begin{proof}
The recurrence relation \[a_{s}=4a_{s-1}+1, \] with initial condition $a_1=r$ must be solved in order to apply the recursive congruence relation (\ref{R9-rec}) of Theorem \ref{nine-regular}.

Let $F(x)=a_1x+a_2x^2+\cdots $ be the generating function for the sequence $a_1,a_2,\ldots$. Then the above recurrence relation yields
\[F(x)=\frac{(r-1)x}{1-4x}+\frac{x(1-x)}{1-4x}.
\]
We obtain $a_s=(r-1)4^{s-1}+\frac{4^s-1}{3}$ when we expand the aforementioned expression using the geometric series expansion. 

We now obtain \[R_9(a_s)\equiv R_9(a_{s-1})\pmod{3}\] according to Theorem \ref{nine-regular}.
The preceding recursive congruence is then applied $s-1$ times, yielding the following congruence: 
\[R_9(a_s)\equiv R_9(r)\pmod{3}.\]
Now (\ref{r9}) follows. After fixing $r=1, r=2$, and $r=3$, respectively, and noticing that $R_9(1)=1$, $R_9(2)=2$, and $R_9(3)=3$, we obtain the particular cases.
\end{proof}
\begin{theorem}
For every prime number $p$, we have
\begin{equation}
R_{p}(n)\equiv \tau_{p-1}(n+1)\pmod{p}.
\end{equation}
\end{theorem}
\begin{proof}
Since $p$ is a prime number, for $1\leq k\leq \lfloor\frac{p}{2}\rfloor$, we obtain
\begin{align*}{{p-1}\choose{k}}&=\frac{tp+k!(-1)^k}{k!}\\
&\equiv (-1)^{k}\pmod{p}.
\end{align*}
Given this observation, we may write
\begin{align*}
\sum_{n=0}^{\infty}\tau_{p-1}(n+1)q^{n}&\equiv\prod_{m=1}^{\infty}\left(1+q^{m}+\cdots +q^{(p-1)m}\right)\pmod{p}\\
&\equiv\prod_{m=1}^{\infty}\frac{1-q^{pm}}{1-q^m}\pmod{p}\\
&\equiv\sum_{n=1}^{\infty}R_p(n)q^n\pmod{p}.
\end{align*}
Now the result follows.
\end{proof}
\section{ Divisibility of Partition Function Weighted Composition Sums }
The Ramanujan's congruences for the partition function are as follows:
\begin{description}
\item{(a)} $p(5n+4)\equiv 0\pmod{5}$,
\item{(b)} $p(7n+5)\equiv 0\pmod{7}$,
\item{(c)} $p(11n+6)\equiv 0\pmod{11}$.
\end{description}
It is rare to find such a simple congruence for other larger primes. The following congruence (modulo $l$, an odd prime), involving partition function values, is found in this section:
\begin{equation*}
\sum_{\substack{{n=a_1+a_2+\cdots a_k}\\{a_{i}\in\mathbb{N}\cup\{0\}}}}p(a_{1})p(a_{2})\cdots p(a_{k})\equiv \sum_{n=t+ls}\tau_{l-k}(t)p(s)\pmod{l}.
\end{equation*}
The following lemma is crucial in obtaining the above congruence.
\begin{lemma}\label{binomialmodulooddprime}
Let $l$ be an odd prime number. Let $k$ be an integer such that $1\leq k<l$. Let $n\geq 0$ be an integer. Let $r$ be the remainder obtained by division of $n$ into $l$. We have
\begin{equation}
{{n+k}\choose{k}}\equiv \begin{cases}(-1)^{r}{{l-k-1}\choose{r}}\pmod{l}&\text{ when $r\in\{0,1,\cdots ,l-k-1\}$;}\\
0\pmod{l}&\text{when $r\in\{l-k,\cdots ,l-1\}$.}
\end{cases}
\end{equation}
\end{lemma}
\begin{proof}
Consider the following expression:
\[{{n+k}\choose{k}}=\frac{(n+k)(n+k-1)\cdots (n+1)}{k!}.
\]
Using this representation, we find an expression for ${n+k}\choose{k}$ modulo $l$. 

Assume that $n=sl+r$ for some $r\in\{l-1,l-2,\cdots ,l-k\}$. Then $(n+k)(n+k-1)\cdots (n+1)\equiv 0\pmod{l}$. Since $l$ is a prime number and $k<l$, we obtain $\frac{(n+k)(n+k-1)\cdots (n+1)}{k!}\equiv0\pmod{l}$. That is, ${{n+k}\choose{k}}\equiv 0\pmod{l}$ when $r\in\{l-k,\cdots ,l-1\}$. The second case follows. 

Assume that $n=sl+r$ for some $r\in\{0,1,\cdots ,l-k-1\}$. Since $l$ is a prime number and $k<l$, we obtain
\begin{align*}
{{n+k}\choose{k}}&=\frac{(sl+r+k)(sl+r+k-1)\cdots (sl+r+1)}{k!}\\
&=\frac{s^{*}l}{k!}+\frac{(r+k)(r+k-1)\cdots (r+1)}{k!}\\
&\equiv {{r+k}\choose {k}}\pmod{l}\\
&\equiv{{r+k}\choose{r}}\pmod{l}.
\end{align*}
On the other hand, we have
\begin{align*}
(-1)^r{{l-k-1}\choose{r}}&=(-1)^r\frac{(l-k-1)(l-k-2)\cdots (l-k-1-(r-1))}{r!}\\
&=(-1)^r\frac{(l-(k+1))(l-(k+2))\cdots (l-(k+r))}{r!}\\
&=(-1)^r\left(\frac{s^{'}l}{k!}+(-1)^r\frac{(k+1)(k+2)\cdots (k+r)}{r!}\right)\\
&\equiv (-1)^r(-1)^r{{r+k}\choose{r}}\pmod{l}.
\end{align*}
From the above two congruences, we have the relation
\[{{n+k}\choose{k}}\equiv(-1)^r{{l-k-1}\choose{r}}\pmod{l}
\]
when $k<l$. The proof is now completed.
\end{proof}
Now we are equipped to prove the following main result of this section. 
\begin{theorem}\label{compositionsum}
Let $k\geq 2$ be a positive integer. Let $l> k$ be an odd prime number. Let $p(n)$ be the number of partitions of $n$. We have
\begin{equation}
\tau_{-k}(n+1)\equiv\sum_{n=t+ls}\tau_{l-k}(t)p(s)\pmod{l}.
\end{equation}
In another notation, 
\begin{equation}
\sum_{\substack{{n=a_1+a_2+\cdots a_k}\\{a_{i}\in\mathbb{N}\cup\{0\}}}}p(a_{1})p(a_{2})\cdots p(a_{k})\equiv \sum_{n=t+ls}\tau_{l-k}(t)p(s)\pmod{l}.
\end{equation}
\end{theorem}
\begin{proof}
We have
\begin{align*}
\prod_{m=1}^{\infty}(1-q^m)^{-k}&=\sum_{n=0}^\infty\tau_{-k}(n+1)q^n\\
&=1+\sum_{n=1}^{\infty}\left(\sum_{\substack{{n=a_1+a_2+\cdots a_k}\\{a_{i}\in\mathbb{N}\cup\{0\}}}}p(a_{1})p(a_{2})\cdots p(a_{k})\right)q^n.
\end{align*}
This equality permits us to present the congruence of this result in two different forms. Now in accordance with Lemma \ref{binomialmodulooddprime}, we may write
\begin{align*}
\prod_{m=1}^{\infty}(1-q^m)^{-k}&=\prod_{m=1}^{\infty}\left(\sum_{n=0}^{\infty}{{n+k-1}\choose{k-1}}q^{mn}\right)\\
&\equiv\prod_{m=1}^{\infty}\left[\left(\sum_{r=0}^{l-k}(-1)^r{{l-k}\choose{r}}q^{rm}\right)\left(1-q^{lm}\right)^{-1}\right]\pmod{l}\\
&\equiv\prod_{m=1}^{\infty}\left[(1-q^m)^{l-k}\left(1-q^{lm}\right)^{-1}\right]\pmod{l}.
\end{align*}
Now the result follows.
\end{proof}
\begin{corollary}
Let $p(n)$ be the number of partitions of $n$. We have
\begin{equation}
\sum_{\substack{{a+b=n}\\a,b\in\mathbb N\cup\{0\}}}p(a)p(b)\equiv \sum_{n=t+3s}\omega(t)p(s)\pmod{3},
\end{equation}
where $\omega(t)$ is given in (\ref{PNT}).

\end{corollary}
\begin{proof}
Fix $l=3$ and $k=2$ in Theorem \ref{compositionsum}. Now the result follows as a consequence of the observation $\tau_1(n)=\omega(n)$.
\end{proof}
Let $l\geq 5$ be a prime number. In view of Identity (\ref{JTPI}), we find a sufficient condition for the following congruence:
\[\sum_{\substack{{n=a_1+a_2+\cdots +a_{l-3}}\\{a_{i}\in\mathbb{N}\cup\{0\}}}}p(a_{1})p(a_{2})\cdots p(a_{l-3})\equiv 0\pmod{l}.
\]
\begin{corollary}
Let $l\geq 5$ be a prime number. For non-zero integer $n$, let $R(n, l)$ denote the remainder obtained by division of $n$ into $l$. If $m$ is a non-negative integer such that $m\not\equiv s\pmod{l}$ for every $$s\in\big\{0\big\}\cup\bigg\{R\left((-1)^r{{l-3}\choose{r}}, l\right):0\leq r\leq\frac{l-3}{2}\bigg\},$$ then for $n\equiv m\pmod{l}$ we have
\begin{equation}
\sum_{\substack{{n=a_1+a_2+\cdots a_{l-3}}\\{a_{i}\in\mathbb{N}\cup\{0\}}}}p(a_{1})p(a_{2})\cdots p(a_{l-3})\equiv 0\pmod{l}.
\end{equation}
\end{corollary}
\begin{proof}
Assume the following: 
\begin{enumerate}
\item $m\not\equiv s\pmod{l}$ for every $s\in\{0\}\cup\bigg\{R\left((-1)^r{{l-3}\choose{r}},l\right):0\leq r\leq\frac{l-3}{2}\bigg\}$,
\item $n\equiv m\pmod{l}$ ,
\item $k=l-3$. 
\end{enumerate}
Upon the assumption 3 and in view of Theorem \ref{compositionsum}, we can write
\begin{equation}\label{eqn1-compsum}
\sum_{\substack{{n=a_1+a_2+\cdots a_{l-3}}\\{a_{i}\in\mathbb{N}\cup\{0\}}}}p(a_{1})p(a_{2})\cdots p(a_{l-3})\equiv \sum_{n=t+ls}\tau_{3}(t)p(s)\pmod{l}.
\end{equation}
In view of Identity (\ref{JTPI}), we can write
\[\tau_3(t)=\begin{cases}
(-1)^j(2j+1)&\text{ if $t=\frac{j(j+1)}{2}$;}\\
0&\text{ otherwise.}
\end{cases}
\]
So (\ref{eqn1-compsum}) can be written upon the assumption 2 as follows:
\begin{align*}
\sum_{\substack{{n=a_1+a_2+\cdots a_{l-3}}\\{a_{i}\in\mathbb{N}\cup\{0\}}}}p(a_{1})p(a_{2})\cdots p(a_{l-3})&\equiv \sum_{n={{j+1}\choose{2}}+ls}(-1)^{j}(2j+1)p(s)\pmod{l}\\
&\equiv\sum_{lb+m={{j+1}\choose{2}}+ls}(-1)^{j}(2j+1)p(s)\pmod{l}
\end{align*}
for some non-negative integer $b$. The index of the right extreme summation is non-empty only when ${{j+1}\choose{2}}\equiv m\pmod{l}$ for some $j$. In view of Lemma \ref{binomialmodulooddprime}, $R\left({{j+1}\choose{2}},{l}\right)$ for $j=1,2\cdots $, constitute a subset of the set $\{0\}\cup\bigg\{R\left((-1)^r{{l-3}\choose{r}},{l}\right): 0\leq r\leq\frac{l-3}{2}\bigg\}$. 
By assumption 3 we have $m\not\equiv s\pmod{l}$ for every $s\in\{0\}\cup\bigg\{R\left((-1)^r{{l-3}\choose{r}}, {l}\right):0\leq r\leq\frac{l-3}{2}\bigg\}$. Given this observation, it follows that the index of the right extreme sum is empty. This leads to the conclusion that 
\[\sum_{\substack{{n=a_1+a_2+\cdots a_{l-3}}\\{a_{i}\in\mathbb{N}\cup\{0\}}}}p(a_{1})p(a_{2})\cdots p(a_{l-3})\equiv 0\pmod{l}.
\]
Now the proof is completed.
\end{proof}
An interplay of Lemma \ref{binomialmodulooddprime} with Theorem \ref{taumodulo7} yields the following result of Ramanujan.
\begin{corollary}[Ramanujan \cite{ramanujan-20}]
Let $n$ be a positive integer. We have
\begin{equation*}
\tau(7n)\equiv 0\pmod{7}.
\end{equation*}
\end{corollary}
\begin{proof}
In accordance with Theorem \ref{taumodulo7}, we can write
\begin{align}
\notag \tau(7n)&=\tau(7n-1+1)\\
&\label{mod7}\equiv\sum_{7n-1=\frac{m(m+1)}{2}+7\frac{r(r+1)}{2}}(-1)^{m+r}(2m+1)(2r+1)\pmod{7}.
\end{align}
The index of the summation above suggests that 
\[{{m+1}\choose{2}}={{m-1+2}\choose{2}}\equiv 6\pmod{7}.
\]
Since $6={4\choose{2}}$, in view of Lemma \ref{binomialmodulooddprime}, we obtain that $m-1\equiv 2\pmod{7}$. This implies that $2m+1\equiv 0\pmod{7}$. Substituting this congruence in the sum in (\ref{mod7}), we obtain that $\tau(7n)\equiv 0\pmod{7}$. 
\end{proof}
\noindent {\bf Acknowledgement.} The authors thank the referee and managing editor for their insightful suggestions, which improved the paper's content and presentation. 

\end{document}